\documentclass[EJP]{ejpecp} % replace ECP by EJP if needed.
\usepackage[T1]{fontenc}
\usepackage[utf8]{inputenc}
\usepackage{tikz}

\usetikzlibrary{positioning,arrows.meta} 
\makeatletter
\def\namedlabel#1#2{\begingroup
    #2%
    \def\@currentlabel{#2}%
    \phantomsection\label{#1}\endgroup
}
\makeatother

\usepackage{enumerate}  % uncomment to use this package

\SHORTTITLE{Stochastic Sandpile Model: exact sampling and complete graph}

\TITLE{Stochastic Sandpile Model: exact sampling and complete graph} % \thanks is optional. Insert line breaks with \\

\AUTHORS{%
Concetta~Campailla\footnote{Sapienza Universit\`a di Roma, Dipartimento di Matematica Guido Castelnuovo, Roma, Italy.
\EMAIL{mariaconcetta.campailla@uniroma1.it}}\orcid{0009-0008-7638-4537}
  \and %% remove this line and below if single author
  Nicolas~Forien\footnote{CEREMADE, CNRS, Universit\'e Paris-Dauphine, Universit\'e PSL, 75016 Paris, France. \BEMAIL{forien@ceremade.dauphine.fr}}\orcid{0000-0002-3146-644X}}%AUTHORS
\KEYWORDS{Stochastic Sandpile Model;
self-organized criticality; interacting particle systems} % Separate items with ;

\AMSSUBJ{60K35; 60J27; 82C22} % Edit. Separate items with ;
\SUBMITTED{July 14, 2025} % Edit.
\ACCEPTED{February 19, 2026} % Edit.

\VOLUME{0}
\YEAR{2023}
\PAPERNUM{0}
\DOI{10.1214/YY-TN}

\ABSTRACT{We study the dynamics of the Stochastic Sandpile Model on finite graphs,
with two main results.
First, we describe a procedure to exactly sample from the stationary
distribution of the model in all connected finite graphs, extending a
result obtained by Levine and Liang for Activated Random Walks.
Then, we study the model on the complete graph with a number of vertices
tending to infinity and show that the stationary density tends to~$1/2$.
Along the way, we introduce a new point of view on the dynamics of the
model, with active and sleeping particles, which may be of independent
interest.}

\begin{document}

%%%%%%%%%%%%%%%%%%%%%%%%%%%%%%%%%%%%%%%%%%%%%%%%%%%%%%%%%%%%%%%%%%%
%%                                                               %%
%% No need for \maketitle.                                       %%
%%                                                               %%
%%%%%%%%%%%%%%%%%%%%%%%%%%%%%%%%%%%%%%%%%%%%%%%%%%%%%%%%%%%%%%%%%%%

%%%%%%%%%%%%%%%%%%%%%%%%%%%%%%%%%%%%%%%%%%%%%%%%%%%%%%%%%%%%%%%%%%%
%%                                                               %%
%% Please replace what follows by the body of your article       %%
%% (up to the bibliography):                                     %%
%%                                                               %%
%%%%%%%%%%%%%%%%%%%%%%%%%%%%%%%%%%%%%%%%%%%%%%%%%%%%%%%%%%%%%%%%%%%

\section{Introduction}
 The Stochastic Sandpile Model (SSM) was introduced to study \textit{self-organized criticality},  a notion  first proposed in the eighties~\cite{BTW} to explain the
 critical behaviour in steady states without fine-tuning of system parameters.
 The Stochastic Sandpile Model is a continuous-time
 particle system on a finite or infinite graph, where the number of particles at
 a site~$x$ is denoted by~$\eta(x).$ A site~$x$ is declared
 unstable when~$\eta(x) \geq 2$; otherwise, it is declared stable. Each unstable site~$x$ topples at rate~$1$, sending~$2$
 of its particles to neighbours chosen independently and uniformly.  
 Any configuration with only stable sites is an absorbing state for the dynamics.
 Initially, the number of
 particles at each vertex of the graph is an independent Poisson random variable of parameter~$\rho >0$
 usually called the \textit{particle density}.  
 
 Prior studies of the Stochastic Sandpile Model have primarily focused on its dynamics on~$\mathbb{Z}^d$, or more broadly, on infinite vertex-transitive graphs. In this case, the model undergoes an absorbing-state phase transition. More precisely, Rolla and Sidoravicius~\cite{RS} showed that there exists a critical value~$\rho_c$ such that if~$\rho<\rho_c$, the system fixates, i.e., each vertex is
 visited only finitely many times almost surely. In contrast, if~$\rho>\rho_c$, each vertex is visited infinitely many times almost surely. 
  It was proved that the critical value~$\rho_c$ is strictly below one and strictly greater than zero in any dimension~\cite{CFT, HHRR22, RS, ST17}. 

  The original models of self-organized
 criticality were defined on finite graphs. Consider the dynamics defined  on a large finite subset of an infinite connected graph. Whenever a particle moves to the boundary, it is removed from the system. This is called \textit{dissipation}. Since the set of vertices is finite, almost surely, after a finite time the system globally stabilises, i.e., all sites are stable. At this point, a new particle is added to a random uniform site of the graph.
 This is called \textit{driving}. The new particle can destabilise the system. The dynamics will then continue until a stable configuration is
 reached. Then, a new particle is added again, and one waits until the system globally stabilises, and so on. This procedure defines a Markov chain on the set of stable configurations, which is called the driven-dissipative model.
 When the average particle density is too high, the introduction of one particle triggers intense activity, leading to a significant number of particles being absorbed at the boundary. Conversely, if the average particle density is too low, the particle does not cause much interaction with other particles, causing mass to accumulate instead.  Thanks to this well-structured mechanism, the model is attracted to a stationary
 state with an average density~$\rho_s$ as the finite volume exhausts the graph. 

 It is conjectured that the typical density~$\rho_s$ in the steady state of the driven-dissipative system 
 corresponds to~$\rho_c$, the point at which the infinite system displays the phase transition defined
 above. This is called the \textit{density conjecture} (see for example~\cite{LS24} for a presentation). 
 This problem has been studied for similar models known as Activated Random Walks (ARW) and the Abelian Sandpile model (ASM), which share
 some features with SSM, and which were also introduced to study self-organized criticality. In particular, the density conjecture has been proved for ARW on~$\mathbb{Z}$  by Hoffman, Johnson and Junge~\cite{HJJ} and it
was shown to be false on some graphs~\cite{FLW2} for the ASM.

 The present article contains two main results. The first result is a simple algorithm to exactly sample a configuration from the stationary distribution of the driven-dissipative model, on any finite connected graph. This is the content of Theorem \ref{thmstationarydist} below. This extends a previous result which gave a similar description for ARW~\cite{LL}.
 The second result is that the stationary density~$\rho_s$ of the model on the complete graph is equal to~$1/2$, with exponential concentration bounds for the stationary distribution around this value: see Theorem \ref{mainthm2} below. This result had been suggested in~\cite{FM}.
  Recently, a similar result was obtained for ARW, which determines the stationary density in the complete graph~\cite{JMT}. However, obtaining a similar result for the SSM is more complex.

\section{Definitions and main results} 
\subsection{Definition of the driven-dissipative Markov chain} \label{Sectiondrivendissipativemc}

The driven-dissipative Markov chain is defined as follows.
 Let us consider~$G= (V \cup \{s\},E)$, a connected finite graph with a distinguished vertex~$s$ called the ``sink''. This includes the case mentioned above of~$V$ being a finite subset of an infinite graph, by simply identifying all vertices outside~$V$ into a single vertex called the sink. A particle
 configuration can be represented as~$\eta \in \Sigma:= \{0,1\}^V$, where~$\eta(x)$ denotes the number
 of particles at~$x$ and all sites of $V$ are stable. 
 We initiate the system in an arbitrary configuration of only stable sites, $\eta_0 \in
 \Sigma$.
 At each time step~$n\geq 1$ we add to the configuration~$\eta_{n-1}$ a particle at a vertex~$x$ uniformly chosen in 
~$V \cup \{s\}$. If the sink~$s$ is chosen, we let~$\eta_{n}=\eta_{n-1}$.\footnote{Compared to the chain defined by adding on~$V$ instead of~$V\cup \{s\}$, this amounts to introducing a laziness probability of~$\frac{1}{N}$. We make this choice for technical reasons, to make the statement of Lemma  \ref{conn-micr-macr} shorter, but our main result, Theorem \ref{mainthm2}, also holds for the chain obtained by adding on~$V$, since it has the same stationary distribution.} Otherwise, if the vertex chosen was empty,  we set~$\eta_n:=\eta_{n-1}+\mathbb{1}_x$. If~$x$
 already hosted a particle, this site becomes unstable, because it now hosts two particles. Let the system evolve according
 to the stochastic sandpile dynamics defined above, and whenever a
 particle jumps to the sink, it instantaneously disappears from the system. 
 After
 an almost surely finite time, we obtain a stable configuration, and we define~$\eta_n \in \Sigma$ to be this configuration.
 The sequence of random variables~$\eta_0,\eta_1,\eta_2,\,,\dots$ thereby defined is a discrete-time Markov chain 
 on the finite configuration space~$\Sigma$. The chain is irreducible and
 aperiodic, and thus converges to a stationary distribution, which we denote by~$\pi$.

 More generally, instead of choosing the driving site uniformly, one may add the particle at a vertex sampled from an arbitrary probability distribution on $V \cup \{s\}$ which is not supported on~$\{s\}$. This defines a different driven–dissipative Markov chain but, as explained in Section~\ref{sec-exactsamplingresult}, it has the same stationary distribution, so the results of this paper extend to this more general setting.
 
 \subsection{Stationary density on the complete graph} \label{sec-stationary-density}
Let~$N\geq 1$. We denote by~$G_N = (V_{N}\cup \{s_N\},E_N)$ the complete
graph with self-loops, $|V_N|=N-1$ vertices, and with a distinguished vertex~$s_N$ playing the role of the sink. Since the graph is complete, any two stable configurations with the same number of particles are equivalent, so we can represent the configuration of the Markov chain~$(\eta_t)_{t \in \mathbb{N}}$ on~$G_N$ at time~$t$ with the number of particles~$\xi_t=\sum_{x \in V_N} \eta_t(x)$.  Then this new process~$(\xi_t)_{t \in \mathbb{N}}$ is a Markov chain  on~$\{0,\,\dots,\,N-1\}$ and its stationary distribution~$\pi_N$ satisfies the following. 
\begin{theorem}\label{mainthm2}
     For any~$\varepsilon>0$, there exists~$c>0$ such that, for~$N$ large enough,
\begin{align*}
            &\pi_N\big(\left[(\rho_s-\varepsilon)N,(\rho_s+\varepsilon)N\right]\big)\geq 1-e^{-cN},
        \end{align*}
    where
    \begin{equation*}
        \rho_s=\frac{1}{2}.
    \end{equation*}
In particular, $\pi_N$ tends to the Dirac measure at~$\rho_s$ when~$N\to\infty$. 

\end{theorem}
Whatever the initial number of particles, the system converges to a stationary state with a number of particles around the value~$N/2$ when~$N$ tends to infinity.   

Note that the above theorem continues to hold when the driving site is chosen according to an arbitrary probability distribution on $V\cup \{s\}$ which is not supported on~$\{s\}$, because the stationary distribution does not depend on the driving (see Section~\ref{sec-exactsamplingresult}).

 We denote by~$\mathbb{P}^N_k$ and~$\mathbb{E}^N_k$ the law and expectation, respectively, of the Markov
 chain~$(\xi_t)_{t\in \mathbb{N}}$ on~$G_{N}$ when~$\xi_0=k$. 
For any sequence~$(a_t)_{t \in \mathbb{N}}$ we
 use the notation \begin{equation*}
     \Delta a_t:=a_{t+1}-a_t, \ \ \ t \in \mathbb{N}
 \end{equation*}
 for the forward difference operator.

\subsection{Half-toppling dynamics}
\label{sec-half-topp}
One of the differences between ARW and SSM, which makes the latter model more difficult to study, is that in the SSM case, particles always move in pairs. A natural way to simplify the dynamics of the model is to allow a single particle to move;  to achieve this, the particles are sent to neighboring sites, but one at a time, leading to what we call \textit{half-topplings}. However, this must be consistent with the specification of the model. For this reason, sites which have seen an odd number of half-topplings
and contain at least one particle are deemed unstable. The half-toppling operation was introduced in~\cite{RS}.

 In order to characterise the stationary distribution~$\pi$ of the Markov chain~$\eta = (\eta_t)_{t\in \mathbb{N}}$,
 we introduce the Diaconis-Fulton graphical representation for the dynamics of SSM.
  Let~$G=(V\cup\{s\},\,E)$ be a connected finite graph with a sink vertex~$s$. The state of the system
 is denoted by~$(\eta,h)$, where~$\eta: V\to \mathbb{N}$ denotes the number of particles at each site and~$h:V\to \mathbb{N}$  is a function
 which counts the number of jumps already made at each site. A site~$x$ is unstable in the configuration~$(\eta,h)$  if~$\eta(x)\geq 2$ or~$\eta(x) = 1$ and~$h(x)$ is odd, otherwise we say that the vertex is stable. The configuration~$(\eta,h)$ is stable in~$V$ if every site of~$V$ is stable, otherwise we say that~$(\eta,h)$ is unstable in~$V$.
 We now define, for every~$x \in V$ and every~$y \in V \cup \{s\},\, y \neq x$ neighbour of~$x$ in the graph~$G$, an operator~$I_{xy}$ which corresponds to one particle jumping from~$x$ to~$y$: namely, if~$x \in V$ is unstable in the configuration~$(\eta,h)$, we define

 \begin{equation*}
    I_{xy}\eta(z)=
    \begin{cases}
\eta(x)-1
&\quad\text{if }z=x,\\
\eta(y)+1
&\quad\text{if }z=y \ \text{and} \ y \neq s,\\
\eta(z) 
&\quad\text{otherwise.}
\end{cases}
\end{equation*}

Note that if~$G$ contains self-loops and~$y=x$, then~$I_{xx}\eta=\eta$.

 Let us fix for now an array of instructions~$I =
(I^{x,j}\,:\, x \in V,\, j \in \mathbb{N})$,
where for every~$x\in V$ and~$j\in\mathbb{N}$,
~$I^{x,j}$ is a jump instruction~$I_{xy}$ to one of the neighbours~$y$ of~$x$ in the graph.

If~$x \in V$ is unstable in~$(\eta,h)$, we say we half-topple~$x$ when we act on a configuration~$(\eta, h)$ through the operator
$\Phi_x$, which is defined as
\begin{equation*}
    \Phi_x(\eta,h)=(I^{x,h(x)}\eta,h(x)+\mathbb{1}_x),
\end{equation*}
where~$\mathbb{1}_x(z)=1$ if~$x=z$ and~$0$ otherwise. The toppling operation~$\Phi_x$ at~$x$ is termed legal for a configuration
$(\eta, h)$ when~$x$ is unstable in~$(\eta, h)$, and illegal otherwise.
Given a sequence~$\alpha = (x_1,\,x_2,\,\dots,\,x_k)$ of non-sink sites, we denote by~$\Phi_\alpha = \Phi_{x_k} \Phi_{x_{k-1}}\cdots \Phi_{x_2} \Phi_{x_1}$
the
composition of the half-topplings at the sites~$x_1,\, x_2,\dots,\,x_{k-1},\, x_k$, in that order. We call~$\Phi_\alpha$ legal for a
configuration~$(\eta, h)$ if~$\Phi_{x_1}$ is legal for~$(\eta, h)$ and~$\Phi_{x_\ell}$ is legal for~$\Phi_{(x_1,\,x_2,\,\dots,\,x_{\ell-1)}}(\eta, h)$ for each~$\ell = 2,\,\dots,\,k$. In
this case we call~$\alpha$ a legal sequence of half-topplings for~$(\eta, h)$. We say that~$\alpha$ stabilises a configuration~$(\eta, h)$
if every site~$x \in V$ is stable in the configuration~$\Phi_{\alpha}(\eta,h)$.
Let~$m_\alpha=(m_{\alpha}(x): x \in V)$ be given by
\begin{equation*}
m_\alpha(x)=\sum_{\ell=1}^k \mathbb{1}(\{x_\ell=x\})\,.
\end{equation*}
In other words, $m_\alpha(x)$ indicates the number of times that the site~$x$ appears in the sequence~$\alpha$. 
 The following lemma gives one of the
 fundamental properties of the Diaconis-Fulton representation. For the proof, we refer to~\cite{RS}.
 \begin{lemma}[Abelian property]\label{abelianproperty}
 If~$\alpha$ and~$\beta$ are both legal sequences of half-topplings for~$(\eta,h)$ that stabilise~$(\eta,h)$, then~$m_{\alpha}=m_{\beta}$ and~$\Phi_{\alpha}(\eta, h)=\Phi_{\beta}(\eta,h)$.
\end{lemma}
 
The Abelian property tells us that the number of instructions used during the stabilisation of a configuration~$(\eta, h)$ and the final configuration do not depend on the half-toppling sequence used, so we can choose the order
 with which the sites are half-toppled, as long as the half-topplings are legal.
 This allows us to define the stable configuration reached when stabilising~$(\eta,h)$ without having to specify the order of the half-topplings. 
 We now introduce a probability measure on the space of instructions and of configurations. 
Let~$\mathcal{P}$ be a probability measure on the set of all possible fields of instructions,
which is such that the instructions~$(I^{x,j} )_{x\in V, j \in \mathbb{N}}$, are independent, with, for every~$x \in V$ and any~$j\in \mathbb{N}$, $\mathcal{P}(I^{x,j}=I_{xy})=\frac{1}{d_x}$ for any~$y \in V \cup \{s\}$ neighboring~$x$, where
~$d_x$ is the degree of the vertex~$x$ in~$G$. Let us consider~$\mathcal{P}_0$ a probability measure on the set~$\mathbb{N}^V \times \mathbb{N}^V$. We write~$\smash{P= \mathcal{P}_0 \otimes \mathcal{P}}$  for the probability measure
 where the initial configuration~$(\eta_0,h_0)$ and the array of instructions~$I$ are independent and respectively distributed according
 to~$\mathcal{P}_0$ and~$\mathcal{P}$.

\subsection{Exact sampling result}\label{sec-exactsamplingresult}
We can now state our other main result, which gives a procedure to exactly sample from the stationary distribution~$\pi$ of the Markov chain~$\eta=(\eta_t)_{t\geq 0}$ in all connected finite graphs.

\begin{theorem} \label{thmstationarydist}
Let~$G=(V \cup \{s\},E)$ be any finite connected graph with a sink vertex~$s$. Consider an initial configuration~$(\eta_0,h_0)$ such that~$\eta_0(x)=1$ for any~$x \in V$ and~$(h_0(x))_{x \in V}$ are independent Bernoulli random variables of parameter~$1/2$. Perform legal half-topplings in whatever order, with particles being killed when they jump to the sink, until no unstable sites remain in~$V$. Then the resulting stable configuration is exactly distributed according to~$\pi$, the stationary distribution of the driven-dissipative Markov chain~$(\eta_t)_{t \in \mathbb{N}}$ defined in Section \ref{Sectiondrivendissipativemc}. 
\end{theorem}

The idea of this theorem relies on the following elementary observation: if we add one particle to the configuration~$(\eta_0,h_0)$ and make this particle walk with half-topplings until it exits from~$V$, then after this the parities of the odometer at the sites of~$V$ are still i.i.d.\ Bernoulli with parameter~$1/2$. This will be detailed in Section \ref{sec-proof-exactsamp}. 

As hinted above, Theorem \ref{thmstationarydist} remains valid for the driven–dissipative Markov chain in which the added particle is placed according to an arbitrary probability distribution on 
$V \cup \{s\}$ which is not supported on~$\{s\}$.
The same proof applies to this more general setting.
This implies in particular that the stationary distribution does not depend on the choice of this driving probability distribution.

\subsection{A new point of view on half-topplings}\label{sleeping-act not}
 We now introduce a different notation for the half-toppling dynamics defined in Section \ref{sec-half-topp}.  
 The idea is the following.
On the one hand, when we half-topple a site~$x$ such that~$\eta(x)=2$ and~$h(x)$ is even, then after the half-toppling we will have~$\eta(x)=1$ and~$h(x)$ odd, which means that this site remains unstable and we will have to half-topple it again in the future. If, on the other hand, $\eta(x)=2$ and~$h(x)$ is odd, then one half-toppling is sufficient to make~$x$ stable. For this reason, even though in SSM
particles can only be of one type, we can view $x$ as containing two ‘active’ particles in the
first case but as containing one ‘active’ and one ‘sleeping’ particle in the second. We therefore choose to represent a configuration~$(\eta,h)$ by a vector~$w: V \to \{\mathfrak{s}+n: n \in \mathbb{N}\} \cup \mathbb{N}$, with~$n<n+\mathfrak{s}<n+1$ for every~$n \in \mathbb{N}$, 
 that gives the number of active and sleeping particles at each site~$x \in V$ as follows. Assume that~$(\eta,h)$ is such that for every~$x \in V$, if~$\eta(x)=0$ then~$h(x)$ is even (otherwise the couple~$(\eta,h)$ represents a configuration which is impossible to reach with legal half-topplings). If~$\eta(x)=0$, we set~$w(x)=0$ (the site is empty). If~$\eta(x)=1$ and~$h(x)$ is odd, we set~$w(x)=1$, which means that there is one active particle on the site~$x$. When~$\eta(x)=1$ and~$h(x)$ is even, we set~$w(x)=\mathfrak{s}$, meaning that there is a sleeping particle on the site~$x$. More generally, if~$\eta(x)\geq 1$ and~$\eta(x) \equiv h(x) \pmod{2}$, we set~$w(x)=\eta(x)$, which means that there are~$\eta(x)$ active particles on the site~$x$; otherwise, if~$\eta(x)\geq 1$ and~$\eta(x) \not \equiv h(x) \pmod{2}$, we set~$w(x)=\eta(x)-1+\mathfrak{s}$ which means there are~$\eta(x)-1$ active particles and one sleeping particle on the site~$x$.  With this notation, a site is unstable when it contains at least one active particle, i.e., when~$w(x)\geq 1$. When we half-topple an unstable site~$x$, it means that~$w(x)\geq 1$ and an active particle jumps from~$x$ to a neighbour~$y$, where it falls asleep instantaneously, and as soon as a particle arrives at a site that had one sleeping particle (with maybe also active particles), both particles become active immediately.
Thus, with this notation the half-toppling operator~$I_{xy}$ would act on the configuration~$w$ as follows:
 \begin{equation*}
    I_{xy}w(z)=
    \begin{cases}
w(x)-1
&\quad\text{if }z=x,\\
w(y)+\mathfrak{s}
&\quad\text{if }z=y \ \text{and} \ y \neq s,\\
w(z) 
&\quad\text{otherwise,}
\end{cases}
\end{equation*}
with the convention that~$(n+\mathfrak{s})-1=n-1+\mathfrak{s}$, for every~$n\geq 1$, and~$(n+\mathfrak{s})+\mathfrak{s}=n+2$, for every~$n\geq 0$. 

Note that, with this new notation, the initial configuration in Theorem \ref{thmstationarydist} can be seen as each particle being active with probability~$1/2$, and sleeping with probability~$1/2$.

 In Section \ref{sec-micr-chain} we will use this notation to study the half-topplings dynamics in terms of active and sleeping particles.

\subsection{Outline of the paper}
In Section \ref{sec-proof-exactsamp}, we prove Theorem \ref{thmstationarydist} using the Abelian property (Lemma \ref{abelianproperty}).

In Section \ref{sec-micr-chain}, we will relate the evolution of the Markov chain~$(\xi_t)_{t\geq 0}$, which we call the \textit{macroscopic} chain, to the behaviour of an other chain called the \textit{microscopic chain}, which tracks the evolution of the number of active and sleeping particles during the stabilisation which follows the addition of one particle.

In Sections \ref{sec-proof-upperbound} and \ref{sec-proof-lowerbound}, we prove respectively the upper bound and the lower bound in Theorem \ref{mainthm2}. Our proof differs from what was done in~\cite{JMT} for the ARW and uses  the following strategy. Consider the complete graph with~$N-1$ vertices and with a distinguished vertex playing the role of the sink. We want to show that, as~$N \to \infty$, the fraction of occupied sites under the stationary distribution of the model concentrates around~$\frac{1}{2}$. 

One of the main ingredients of our strategy is the following lemma, whose proof is deferred to the Appendix \ref{sec-appendix}. 
\begin{lemma}\label{thmreturntime}
Consider, for every integer~$N\geq 1$, an irreducible time-homogeneous
Markov chain~$(Z_t)_{t\geq 0}$ with state space~$\{0,\,\ldots,\,N-1\}$ and
denote by~$\mathbb{P}_k^N$ its distribution when started from a
fixed~$k\in\{0,\,\ldots,\,N-1\}$, and let~$\nu^N$ be its stationary
distribution.

Let~$0<\alpha<\beta<\gamma\leq 1$ and make the two following hypotheses
on the family of Markov chains:

\begin{itemize}

\item[\namedlabel{hpH1}{(H1)}] (Positive drift below level~$\gamma N$) 
There exists a real random variable~$D$ with~$\mathbb{E}[D]>0$ such that,
for every~$N$ large enough, for every~$k\leq \gamma N$, under~$\mathbb{P}_k^N$,
the first jump~$Z_1-k$ stochastically dominates~$D$.

\item[\namedlabel{hpH2}{(H2)}](Long jumps backwards are rare)
There exists~$c>0$ such that, for every~$N$ large enough, for
every~$k\geq \gamma N$, we have~$\mathbb{P}_k^N(Z_1< \beta N )\leq e^{-cN}$.

\end{itemize}

Then, there exists~$c'>0$ such that, for~$N$ large
enough, $\nu^N\big(\{0,\,\ldots,\,\lfloor\alpha N\rfloor\}\big)\leq e^{-c'N}$.
\end{lemma}

We consider the driven-dissipative Markov chain defined in Section \ref{Sectiondrivendissipativemc} where, at each step~$t\geq 0$, we start with a configuration~$\eta_t \in \{0,1\}^{V_N}$, add one particle at a vertex uniformly chosen in~$V_N$ and let the dynamics evolve until we obtain a stable configuration~$\eta_{t+1}$. We denote by~$\xi_t$ the number of particles after~$t$ steps as defined in Section \ref{sec-stationary-density}.
When there are many particles in the initial configuration, we expect that the addition of a particle will destabilise the system and that, during stabilisation, a certain number of particles will leave the system jumping to the sink. When the number of particles is small enough, we expect that the particles tend to accumulate in the system.  In particular,
in Section \ref{sec-proof-lowerbound} we prove that if we start with~$\xi_0<\frac{N}{2}$ particles, $(\xi_t)_{t\geq 0}$ has a positive drift and the Conditions \ref{hpH1} and \ref{hpH2} are satisfied.
Moreover, in Section~\ref{sec-proof-upperbound} we prove that if we start with~$\xi_0 >\frac{N}{2}$ particles, then~$(\xi_t)_{t\geq 0}$ has a negative drift and the Conditions \ref{hpH1} and \ref{hpH2} are satisfied. For this side of the estimate we will use a reversed version of Lemma \ref{thmreturntime}, that is to say, we apply it to the chain~$N-1-\xi_t$.
Then, by Lemma~\ref{thmreturntime}, we obtain Theorem \ref{mainthm2}.
\section{Comments and perspectives}
\subsection{The model with generic instability threshold}
Consider a more general model in which an integer~$k$ is specified, and a site~$x$ is considered unstable if it contains at least~$k$ particles; otherwise, if~$x$ contains less than~$k$ particles,
it is stable. Each unstable site topples at rate~$1$, sending~$k$ of its particles to neighbours chosen uniformly and independently at random. The half-toppling dynamics (see Section \ref{sec-half-topp}) can be adapted for this model, where a site~$x$ is unstable in the configuration~$(\eta,h)$ if~$\eta(x)\geq k$ or~$\eta(x)\leq k-1$ and~$h(x) \not \equiv 0  \pmod{k}$, stable otherwise. It is easily seen that Theorem \ref{thmstationarydist} also holds for this model with starting configuration~$(\eta_0,h_0)$ such that~$\eta_0(x)=k-1$ for any~$x \in V$ and~$(h_0(x))_{x \in V}$ are independent uniform random variables on~$\{0,\,\dots,\,k-1\}$.

Regarding Theorem \ref{mainthm2},  the situation is different.  In this case what counts is the number of sites that have~$k-1$ particles in the initial configuration~$\eta_0 \in \{0,\,\dots,\,k-1\}^{V_N}$ (and not the total number of particles) because when a particle arrives at such a site it makes it unstable (if we use the notation sleeping/active particles, this means that~$k$ particles are activated). For this reason, what we can show is that the number of active particles has a positive drift if the number of sites with~$k-1$ sleeping particles is at least~$N/k$ (regardless of the number of initial active particles). More generally, we expect that, at stationarity, the number of sites with exactly~$\ell$ sleeping particles concentrates around~$N/k$ for each~$\ell \in \{0\,\dots,\,k-1\}$. This would imply that the stationary density $\rho_s$ in this case is equal to $\frac{k-1}{2}$.

\hspace{0.4cm}
\subsection{An alternative strategy based on exact sampling}
It is worth noting that to prove Theorem \ref{mainthm2} we could have used a method similar to that in~\cite{JMT} for ARW, using the exact
sampling result, Theorem \ref{thmstationarydist}. In this case, we should have started with an initial configuration~$(\eta_0,h_0)$ such that~$\eta_0(x)=1$ for any~$x \in V$ and~$(h_0(x))_{x \in V_N}$ are independent Bernoulli random variables of parameter~$\frac{1}{2}$ and performed legal half-topplings until no unstable sites remain in~$V_N$.
This means studying the stabilisation of the microscopic chain~$(X_t, Y_t)_{t\geq 0}$ (see Section \ref{sec-micr-chain}) when the initial configuration is such that~$X_0=N-1$ and~$Y_0$ is a Binomial random variable with~$N-1$ trials and probability of success~$\frac{1}{2}$. 
As in our strategy, if the total number of particles~$X_t$ and the number of active particles~$Y_t$ are large enough, then we expect that~$Y_t$ has a negative drift (Lemma \ref{prop2}); otherwise, if~$X_t$ is large enough but~$Y_t$ is small enough, then we expect that~$Y_t$ has a positive drift (Lemma \ref{prop1}). Let us denote by~$f(x,\,y)$ the expectation of~$\Delta Y_t$ conditioned on the event~$\{X_t=x,Y_t=y\}$. The idea of~\cite{JMT} is that for ARW, $f(x,\,y)$ is zero if the ratio between~$x$ and~$y$ satisfies a certain linear relation, i.e., if~$y = a x + c$, where~$c < 0$. 

If~$y$ is below this line, then~$f(x,y)$ is positive while if~$y$ is above this line, then~$f(x,y)$ is negative.
We expect the same for SSM. The system remains around the line~$y = a x + c$ until it reaches a stable configuration, when~$y=0$. Thus, Theorem \ref{thmstationarydist} tells us that the critical density of the model should be the value~$x = - c/ a$. For this reason, to find the critical value for SSM with this strategy we should determine the coefficients of the line~$y= a x + c$ for which~$f(x,y) = 0$.

Our alternative approach is more general because it is not necessary to know the exact sampling result.
\subsection{Critical window}
As done in \cite{JMT} for ARW, one could study the critical window around which the density of particles at stationarity concentrates also for SSM, in order to obtain a more precise estimate, but we restricted ourselves to computing the limiting density, to keep the presentation simple.
\section{Proof of Theorem \ref{thmstationarydist}: exact
sampling}\label{sec-proof-exactsamp}

Let~$Q$ be the transition matrix of the Markov chain~$(\eta_t)_{t\geq
0}$. As described in Theorem \ref{thmstationarydist}, we consider
the configuration~$(\eta_0, h_0)$ in which
$\eta_0(x)=1$ for any~$x \in V$ and~$(h_0(x))_{x \in V}$ are
independent Bernoulli random variables of parameter~$1/2$.
We consider~$I$ an array of instructions with distribution~$\mathcal{P}$,
independent of~$h_0$.
Let~$(\eta', h')$ be the configuration obtained after
stabilising~$(\eta_0, h_0)$, and let~$\nu$ be the distribution
of~$\eta'$. 
Our aim is to show that~$\nu=\pi$, the stationary distribution of the
Markov chain~$(\eta_t)_{t\geq 0}$.
To prove this, we want to show that~$\nu Q=\nu$, that is to say, if
$\nu$ is the initial
distribution, after one step of the chain we obtain again the distribution
$\nu$.
In other words, we want to show that if we add one extra particle to
the configuration~$(\eta', h')$ at a uniformly chosen vertex~$x^\star$
and if we call~$(\eta'_1, h'_1)$ the configuration obtained by
stabilising~$(\eta'+\mathbb{1}_{x^\star},h')$ via legal
half-topplings, then the distribution of~$\eta'_1$ (which
is~$\nu Q$ by definition\footnote{Note that by construction of the
odometer~$h'$, we have that~$h'(x)$
is even for every~$x\in V$, so that~$\eta'_1$ is distributed as
the configuration obtained by
stabilising~$(\eta'+\mathbb{1}_{x^\star}, 0)$ with a fresh array of
instructions, and thus indeed has distribution~$\nu Q$.}), is equal to~$\nu$.

First, notice that~$(\eta'_1, h'_1)$ is the configuration obtained when
legally stabilising~$(\eta_0+\mathbb{1}_{x^\star}, h_0)$.
Indeed, the legal sequence of half-topplings performed on~$(\eta_0,
h_0)$ to obtain~$(\eta', h')$ is \textit{a fortiori} also legal
for~$(\eta_0+\mathbb{1}_{x^\star}, h_0)$, and the resulting configuration
is~$(\eta'+\mathbb{1}_{x^\star}, h')$.
In other words, if we first stabilise~$(\eta_0+\mathbb{1}_{x^\star},
h_0)$ but disregarding the extra particle, we
obtain~$(\eta'+\mathbb{1}_{x^\star}, h')$, and if we then stabilise
with the extra particle present, we obtain~$(\eta'_1, h'_1)$.

Now, let us consider a different procedure to
stabilise~$(\eta_0+\mathbb{1}_{x^\star}, h_0)$.
First we only move the extra particle at~$x^\star$ until it leaves the system.
This happens almost surely after a finite number of steps, since the
graph is finite and connected and all vertices are occupied by a
particle.
Denote by~$(\eta_1,h_1)$ the obtained configuration and observe that
$\eta_1=\eta_0$, i.e., there is still one active particle on each
site.
We then stabilise~$(\eta_1,h_1)$.
By the Abelian property (Lemma~\ref{abelianproperty}), starting
from~$(\eta_0+\mathbb{1}_{x^\star},h_0)$, the final configurations
obtained with the two legal procedures of half-topplings are the same.
Therefore, when stabilising~$(\eta_1, h_1)$ we obtain~$(\eta'_1,
h'_1)$.

To sum up, on the one hand, the stabilisation of~$(\eta_0, h_0)$ gives
the configuration~$\eta'$,
which has distribution~$\nu$,
while on the other hand, the stabilisation of~$(\eta_1, h_1)$ gives~$\eta'_1$,
which has distribution~$\nu Q$ (see Figure~\ref{figure1}).
Recall that~$\eta_0=\eta_1$.
Thus, to show that~$\nu=\nu Q$, it suffices to prove that~$ h_1 \pmod{2}$ has the same distribution as
$h_0$.

For any~$x \in V$, let us denote by~$q_0(x) \in \{0,1\}$ the parity of
the number of visits made by the extra particle (which started at~$x^\star$) to site~$x$ before ending up
in the sink, so that~$h_1 \equiv h_0+q_0 \pmod{
2}$.
For any~$h \in \{0,1\}^V$, we have
\begin{align*}
   P\big(h_1 \equiv h \pmod{2} \big)
&= \sum_{ q \in \{0,1\}^V} P\big(h_1 \equiv h \pmod{2},
q_0=q\big)
   \\& =\sum_{ q \in \{0,1\}^V} P\big(h_0 \equiv  h+q  \pmod{2},
q_0=q\big)
   \\&= \sum_{ q \in \{0,1\}^V} P\big(h_0 \equiv h+q \pmod{2}\big)
P\big( q_0=q\big)
\,,
\end{align*}
where the last equality follows from the fact that~$q_0$ is independent of
$h_0$.
Then, using that~$h_0$ is
uniform in~$\{0,1\}^V$, this simplifies to
\[
P\big(h_1 \equiv h \pmod{2} \big)
= \sum_{ q \in \{0,1\}^V} \frac{1}{2^{|V|}}P(q_0=q)  
   =\frac{1}{2^{|V|}}
 =P(h_0=h)
\,.
\]
 Since~$\eta_1=\eta_0$, we have that the configuration~$(\eta_1,h_1 \pmod{2})$ has the same distribution as~$(\eta_0,h_0)$. Thus,
the distribution~$\nu Q$, which is obtained after stabilisation
of~$(\eta_1, h_1)$, is equal to the distribution~$\nu$, which is
obtained after stabilisation of~$(\eta_0, h_0)$.
Since the chain is irreducible, the stationary distribution is unique
and therefore~$\nu=\pi$.

\begin{figure}%[htbp]
\centering
\begin{tikzpicture}[>=Stealth, every node/.style={font=\small}]

  % --- Grafico 2 (a sinistra) ---
  % Centrato su y = 0
 
  \node (B2) at (1,-3) {$(\eta', h')$};
  \node (C2) at (1,1) {$(\eta_0, h_0)$};

  \draw[<-] (B2) -- (C2) node[midway, right] {stabilise};

  % --- Grafico 1 (a destra) ---
  % Anche centrato su y = 0
  \node (A1) at (6,1) {$(\eta_0+\mathbb{1}_{x^\star}, h_0)$};
  \node (B1) at (12,1) {$(\eta_1, h_1)$};
  \node (C1) at (12,-3) {$(\eta'_1,h'_1)$};
  \node (D1) at (6,-3) {$(\eta'+\mathbb{1}_{x^\star}, h')$};

  \draw[->] (A1) -- (B1) node[midway, above] {force the extra particle to leave};
  \draw[->] (B1) -- (C1) node[midway, right] {stabilise};
  \draw[<-] (C1) -- (D1) node[midway, below] {stabilise};
  \draw[<-] (D1) -- (A1) node[midway, above, sloped, align=left] {stabilise, disregarding \\the extra particle};
 
\end{tikzpicture}
\caption{The two different stabilisation procedures of the configuration $(\eta_0+\mathbb{1}_{x^*},h_0)$ used in the proof of Theorem \ref{thmstationarydist}.}
\label{figure1}
\end{figure}
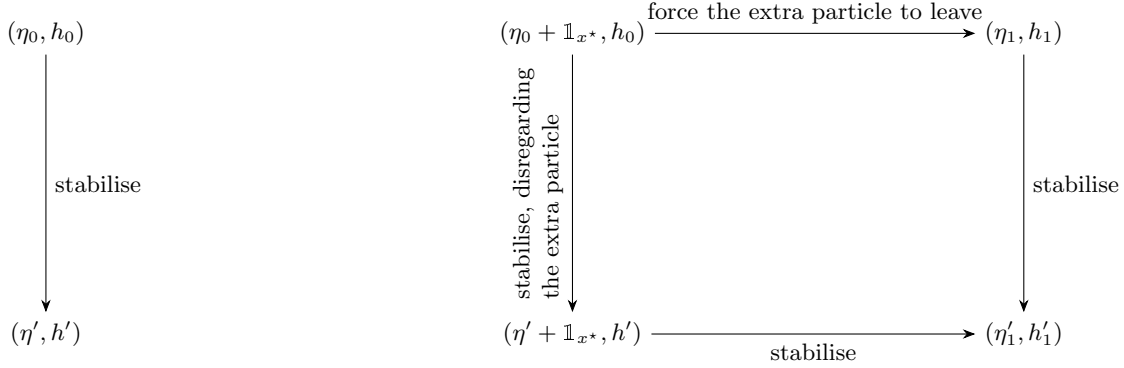

 \section{The microscopic chain of sleeping/active particles}
\label{sec-micr-chain}

Let~$N\geq 1$ and let~$G_N=(V_N \cup \{s_N\}, E_N)$ be the complete
graph with self-loops, $|V_N|=N-1$ vertices and a sink vertex~$s_N$.
Recall that one step of the chain~$(\xi_t)_{t\geq 0}$ defined in
Section \ref{sec-stationary-density} consists of adding one particle
to a uniformly chosen vertex of a stable configuration and using a
legal sequence of half-topplings that stabilises the configuration
obtained.
Let us call this Markov chain~$(\xi_t)_{t\geq 0}$ the \textit{macroscopic chain}.

In this section, we relate one step of this chain to the
behaviour of another Markov chain~$(X_t, Y_t)_{t\geq 0}$ that we call
the \textit{microscopic chain}.
This chain describes the stabilisation after adding one particle,
using the new representation
of the half-toppling dynamics introduced in Section \ref{sleeping-act
not}, with active and sleeping particles.
In this chain~$(X_t, Y_t)_{t\geq 0}$, the variable~$X_t$ is the total
number of particles,
regardless of their state, after~$t$ steps, and~$Y_t$ is the number
of active particles.
This chain is formally defined as follows:

\begin{definition}\label{def-micro-chain}
Fix an integer~$N\geq 1$.
Consider a Markov chain~$(X_t, Y_t)_{t\geq 0}$ with state
space
\[
S=
\big\{
(x,y)\in\mathbb{N}^2\,:\,
y\leq x
\text{ and }
x-y\leq N-1
\big\}
\]
and the following transition probabilities: for every~$(x,y)\in S$
with~$y\neq 0$,
\begin{align}
(x,y)
&\ \longrightarrow\ 
(x-1,\,y-1)
\text{ with probability }
\frac 1 N,
\,\label{jump-to-sink}
\\
(x,y)
&\ \longrightarrow\ 
(x,\,y+1)
\text{ with probability }
\frac {x-y} N,
\,\label{jump-to-sleeping}\\
(x,y)
&\ \longrightarrow\ 
(x,\,y-1)
\text{ with probability }
\frac {N-1-x+y} N,
\,\label{jump-to-active}
\end{align}
and the states~$(x,y)\in S$ with~$y=0$ being absorbing states.
We call this chain~$(X_t,Y_t)_{t\in \mathbb{N}}$ the \textit{microscopic chain}.
Given~$k,j \in \mathbb{N}$, we denote by~$P^N_{k,j}$ and~$E^N_{k,j}$ the law
and expectation, respectively, relative to this chain when started at~$X_0=k,\,Y_0=j$.
\end{definition}

The connection between the macroscopic chain and the microscopic chain is formalised in the following result.

\begin{lemma} \label{conn-micr-macr}
For every integer~$N\geq 1$, the jump probabilities of the macroscopic
chain are given by:
\[
\forall k\in\{0,\,\ldots,\,N-1\} \quad \forall \ell \in \{0,\,\ldots,\, k+1\}
\qquad 
\mathbb{P}_k^N(\xi_1=\ell)
\ =\ 
P_{k+1,1}^N(X_{\tau_0}=\ell)
\,,
\]
where~$\tau_0$ is the stopping time defined by~$\tau_0=\inf\{t\geq 0: Y_t=0\}$.
\end{lemma}

The idea of Lemma~\ref{conn-micr-macr} is the following. A single step of the macroscopic chain starting from a configuration with $k$ particles consists in adding one particle at a uniformly chosen site and then letting the system evolve until stabilisation.
This is equivalent to starting the microscopic chain from a configuration with $k+1$ particles of which exactly one is active and letting the dynamics evolve until no active particles remain. Indeed, starting from such a configuration, the first step of the microscopic chain consists in making this unique active particle jump to a uniformly chosen site.

In particular, if we had defined the macroscopic chain by adding the
particle on a uniform site of~$V$ instead of a uniform site
in~$V\cup\{s\}$, then in equation~\eqref{one-step-macroscopic} below we would not
recover exactly the
stabilisation of the microscopic chain~$(X_t,Y_t)_{t\geq 0}$ with
initial configuration~$X_0=k+1$ and~$Y_0=1$, since during the first step the unique active particle could jump directly to the sink.
This is the reason for our choice to add the particle on~$V\cup\{s\}$,
so that this added particle coming from nowhere behaves as if it was
an active particle coming from the configuration itself, as explained above.
\begin{proof}
Let~$N\geq 1$ and let~$k \in \{0,\,\dots,\,N-1\}$ and~$\ell \in
\{0,\,\ldots,\, k+1\}$.
Let us show that
\begin{align}
\begin{split}
\label{one-step-macroscopic}
\mathbb{P}_k^N(\xi_1=\ell)
\ & =\ 
\frac 1 N
\mathbb{1}(\ell=k)
+\frac {N-1-k} N
\mathbb{1}(\ell=k+1)
+\frac k N
P_{k+1,2}^N(X_{\tau_0}=\ell)
\\& =\ 
P_{k+1,1}^N(X_{\tau_0}=\ell) 
\,.
\end{split}
\end{align}
First observe that the second equality in~\eqref{one-step-macroscopic} simply
follows from Markov's property applied to the microscopic chain at
time~$t=1$.
Hence, we only have to show the first equality.

Recall that, as defined
in Section \ref{Sectiondrivendissipativemc}, one step of the
macroscopic chain begins by adding one particle at a vertex uniformly
chosen in~$V_N \cup \{s_N\}$.
Starting with~$k$ sleeping particles, with probability~$\frac{1}{N}$
(resp.~$\frac{N-1-k}{N}$) this vertex
is the sink (resp.\ this vertex was empty) and thus we have~$\xi_1 =k$
(resp.~$\xi_1=k+1$).
This explains the first two terms in~\eqref{one-step-macroscopic}.

Otherwise, with probability~$\frac{k}{N}$, this vertex already hosted a
particle, thus we obtain one unstable site with two particles, and we
let the dynamics evolve until we obtain a new stable
configuration.
Using the representation
introduced in Section~\ref{sleeping-act not}, this means that we start
with~$k-1$ sites containing one sleeping
particle, and one site containing two active particles (because this
site contains two particles and has zero odometer).
Hence, we start with~$X_0=k+1$ and~$Y_0=2$.

Then, to obtain the third term in~\eqref{one-step-macroscopic}, there
only remains to
show that, during the stabilisation of this configuration using an
arbitrary legal half-toppling procedure, the number of
particles~$X_t$ and the number of active particles~$Y_t$ evolve
according to the microscopic chain defined above.

At each step~$t\geq 0$ of the stabilisation procedure, if there
are~$X_t$ particles of whom~\smash{$Y_t>0$} are active, an
active particle jumps to a uniformly chosen site.

\begin{itemize}
\item
With probability~$1/N$, this site is the sink, in which case both the
number of active particles and the total number of particles decrease
by~$1$, whence the jump probability~\eqref{jump-to-sink}.

\item
With probability~$(X_t-Y_t)/N$, the particle jumps to a site occupied
by a sleeping particle (and
possibly by active particles as well), and both particles are woken up, as
explained in Section~\ref{sleeping-act not}.
Thus, the number of active particles increases by~$1$,
whence~\eqref{jump-to-sleeping}.

\item
With probability~$(N-1-X_t+Y_t)/N$, the particle jumps to a non-sink
vertex that does not contain sleeping particles (but may or may not
contain active particles) where it falls asleep instantaneously.
Thus, the number of active particles decreases by~$1$.
This gives the last transition probability~\eqref{jump-to-active}.
\end{itemize}

Thus, we conclude that the total number of particles and the number of
active particles indeed evolve according to the microscopic Markov chain
given in Definition~\ref{def-micro-chain}, which concludes the proof
of the first equality in~\eqref{one-step-macroscopic}, and thereby of
the lemma.
\end{proof}

\section{Proof of the upper bound in Theorem \ref{mainthm2}}\label{sec-proof-upperbound}

We wish to apply Lemma \ref{thmreturntime} to the chain~$(\xi_t)_{t\geq 0}$ with~$\frac{1}{2}<\gamma<\beta< 1$. More precisely, we use a reversed version of Lemma \ref{thmreturntime}, or, in other words, we apply Lemma \ref{thmreturntime} to the chain~$(N-1-\xi_t)_{t\geq 0}$. Since the upward steps of~$\xi_t$ are at most of size~$1$,  if~$\xi_0\leq \gamma N$, then~$\xi_1\leq \gamma N+1<\beta N$, for~$N$ large enough. Thus, the Condition \ref{hpH2} is satisfied.

We now verify Condition \ref{hpH1}. That is, we aim to find a real random variable~$D$ with negative mean such that, for every~$N$ large enough and for any~$k \geq \gamma N$, then under~$\mathbb{P}^N_k$ the first jump~$\xi_1-k$ is stochastically dominated by~$D$.
The following lemma tells us that, if we start with ~$k \geq\gamma N$ particles, with positive probability at least~$\delta N$ particles, for some~$\delta \in (0,\gamma-\frac{1}{2})$, leave the system in one step of the chain~$(\xi_t)_{t\geq 0}$.
\begin{lemma}\label{propupperbound2}
For any~$\gamma>\frac{1}{2}$ and for any~$\delta \in (0,\gamma-\frac{1}{2})$, there exists~$C>0$ such that for any~$N$ large enough and for any~$k\geq \gamma N$, we have 
\begin{equation}
\mathbb{P}^N_k\left(\xi_1\leq k-\delta N\right)>C.\label{equpperbound}
\end{equation}

\end{lemma}

This lemma implies Condition \ref{hpH1}, because, if~$C$ is given by the lemma and~$N_0$ is large enough so that \eqref{equpperbound} holds for every~$N\geq N_0$ and so that~$1-C-C\delta N_0<0$, then the Condition \ref{hpH1} holds for every~$N\geq N_0$ with drift variable~$D$ given by 
\begin{equation*}
    P(D=-\delta N_0)=C \quad \text{and} \quad P(D=1)=1-C,
\end{equation*}
using the fact that~$\xi_1-\xi_0\leq 1$ almost surely.

We deduce that the Conditions \ref{hpH1} and \ref{hpH2} necessary to apply Lemma \ref{thmreturntime} are satisfied. Then, using this lemma we have that for any~$\alpha>\frac{1}{2}$, there exists~$c>0$ such that, $\pi_N\left(\{\lfloor\alpha N \rfloor,\,\dots,\,N\}\right)\leq e^{-cN}$. Thus, to obtain the upper bound in Theorem \ref{mainthm2} there only remains to show Lemma \ref{propupperbound2}.

\subsection{Proof of Lemma \ref{propupperbound2}: many particles exit in one step}\label{sec-proof-propupperbound}
    
We want to show that if we start with~$k\geq \gamma N$ particles, then with positive probability in one step of the chain~$(\xi_t)_{t\geq 0}$ at least~$\delta N$ particles leave the system. Let us consider the dynamics~$(X_t,Y_t)_{t\geq 0}$ defined in Section \ref{sec-micr-chain}. 
By Lemma \ref{conn-micr-macr}, we set~$X_0=k+1,\,Y_0=1$ and let the system evolve according to the dynamics described in Section \ref{sec-micr-chain} up to time~$\tau_0$.
We are going to show that when~$X_t$ is large enough and~$Y_t$ is not too large, the number of active particles~$Y_t$ has a positive drift. In particular, if initially the total number of particles is large enough and the number of active particles is sufficiently small, we expect that, with positive probability, during the procedure stopped at time~$\tau_0$ the jumps of active particles imply the activation of many sleeping particles or the exit of at least~$\delta N$ particles from the system. This idea is stated in Lemma \ref{prop1}.

Fix~$\gamma>\frac{1}{2},\,\delta \in (0, \gamma-\frac{1}{2})$ and define~$\tau^\star=\inf\{t\geq 0: X_t\leq \left(\gamma-\delta\right)N\}$ with the convention that~$\tau^\star=\infty$ if this set is empty. Let us define, for any~$j \geq 0$, \smash{$\tau_{\geq j}=\inf\{t\geq 0: Y_t\geq j\}$}.

\begin{lemma} \label{prop1} 
For any~$\gamma>\frac{1}{2},\,\delta \in (0,\gamma-\frac{1}{2})$ and for any~$\varepsilon\in (0, \gamma-\delta-\frac{1}{2})$, there exists a constant~$c>0$ such that for any~$N$ large enough and~$k\geq \gamma N$,
    \begin{equation*}
        P^N_{k,1}\left(\{\tau_{\geq \varepsilon N}<\tau_0\} \cup \{\tau^\star<\tau_0\}\right)>c.
    \end{equation*}
\end{lemma}

If during the procedure less than~$\delta N$ particles have left the system and there are at least~$\varepsilon N$ active particles, then we expect that, with large probability, many particles exit the system before reaching a stable configuration. 

\begin{lemma} \label{prop2}
For any ~$\gamma>\frac{1}{2},\,\delta \in (0,\gamma-\frac{1}{2})$ and~$\varepsilon\in (0,\gamma-\delta-\frac{1}{2})$, there exists a constant~$c>0$ such that for any~$N$ large enough and~$k\geq (\gamma-\delta) N$, 
    \begin{equation*}
        P^N_{k,\lceil \varepsilon N\rceil}( \tau^\star>\tau_0)\leq e^{-c N}.
    \end{equation*}
\end{lemma}

We postpone the proofs of these two lemmas to Section \ref{sec-proof-tecn-lemmas} and now show how Lemma \ref{propupperbound2} follows from them.
\begin{proof}[Proof of Lemma \ref{propupperbound2}]
 Let~$\gamma>\frac{1}{2},\, \delta \in (0,\gamma-\frac{1}{2})$ and~$\varepsilon \in (0,\gamma-\delta-\frac{1}{2})$.  Let~$k\geq \gamma N$. By Lemma~\ref{conn-micr-macr}, we have
    \begin{equation*}
        \mathbb{P}^N_k\left(\xi_1\leq  \left(\gamma-\delta\right)N\right)= P^N_{k+1,1}\left(\tau^\star\leq \tau_0\right).
    \end{equation*}   
By Lemma~\ref{prop1}, there exists~$c>0$ such that
\[
     c<P^N_{k+1,1}(\{\tau_{\geq \varepsilon N}<\tau_0\} \cup \{\tau^\star\leq \tau_0\})= P^N_{k+1,1}(\tau^\star\leq \tau_0)+ P^N_{k+1,1}(\tau_{\geq \varepsilon N}<\tau_0<\tau^\star).
\]
%If~$P^N_{k,2}(\tau^\star<\tau_0)>c/2, $ we have concluded.
%Otherwise, 
Using Lemma \ref{prop2} we have that there exists~$c'>0$ such that
\[
   P^N_{k+1,1}(\tau_{\geq \varepsilon N}<\tau_0<\tau^\star)\leq \inf_{k'\geq \left(\gamma-\delta \right)N} P^N_{k', \lceil \varepsilon N \rceil}( \tau^\star>\tau_0)\leq e^{-c' N}.
\]
Thus, we can conclude that, for any~$N$ large enough and for any~$k\geq \gamma N$,
\begin{equation*}\label{eq2}
    P^N_{k+1,1}(\tau^\star\leq \tau_0)>c-e^{-c'N}>c/2,
\end{equation*}
which concludes the proof.
\end{proof}

\subsection{Proofs of Lemmas \ref{prop1} and \ref{prop2}: positive drift for the number of active
particles}\label{sec-proof-tecn-lemmas}

Let us now prove Lemma \ref{prop1}, whose basic idea is as follows. If we start from a configuration with a sufficiently large number of particles and only one active particle, then the number of active particles tends to increase over time, showing a positive drift. A martingale argument completes the proof.

    \begin{proof}[Proof of Lemma \ref{prop1}]
    Let~$\gamma>\frac{1}{2},\, \delta \in (0,\gamma-\frac{1}{2})$ and~$\varepsilon \in (0,\gamma-\delta-\frac{1}{2})$. Let us show that if~$c$ is small enough and if~$X_0=k\geq\gamma N$ and~$Y_0=1$, the process 
\begin{equation*}
    M_t=e^{-cY_t} \mathbb{1}(t<\tau)  \  \ \text{with}\ \ \tau=\min\{\tau^\star, \, \tau_{\geq \varepsilon N}\}
\end{equation*}
 is a supermartingale with respect to the filtration~$(\mathcal{F}_t)_{t \in \mathbb{N}}$ generated by~$(X_t,Y_t)_{t\geq 0}$. We can write 
        \begin{equation*}
            E^N_{k,1}[M_{t+1}\mid \mathcal{F}_t]= M_t E^N_{k,1}[e^{-c\Delta Y_t} \mathbb{1}(t+1<\tau)\mid \mathcal{F}_t] \leq M_t  \mathbb{1}(t<\tau) E^N_{k,1}[e^{-c\Delta Y_t} \mid \mathcal{F}_t],
        \end{equation*}
        where, by \eqref{jump-to-sleeping}, we get
        \begin{align*}
          \mathbb{1}(t<\tau) E^N_{k,1}[e^{-c\Delta Y_t} \mid \mathcal{F}_t]&\leq \mathbb{1}(t<\tau) \left(  \frac{X_t-Y_t}{N} e^{-c}+\left(1-\frac{X_t-Y_t}{N} \right) e^{c}\right) 
            \\& =\mathbb{1}(t<\tau) \left( e^{c}-\left(\frac{X_t-Y_t}{N}\right)\left(e^c-e^{-c}\right)\right).
        \end{align*}
        If~$t<\tau$, then~$\frac{X_t-Y_t}{N}\geq \gamma-\delta -\varepsilon$. Thus, 
        \begin{equation*}
            \mathbb{1}(t<\tau) E^N_{k,1}[e^{-c\Delta Y_t}\mid \mathcal{F}_t]\leq e^c-\left(e^c-e^{-c}\right)\left(\gamma-\delta -\varepsilon \right)
        \end{equation*}
        
       and if we choose~$c=c(\gamma,\delta,\varepsilon)>0$ small enough such that
        \begin{equation*}
            \frac{e^c-1}{e^{ c} -e^{-c}}<\gamma-\delta -\varepsilon,
        \end{equation*}
       which is possible because~$\gamma-\delta-\varepsilon>\frac{1}{2}$ and the left-hand side tends to~$\frac{1}{2}$ when~$c\to 0$,  we have that~$M_t$ is a supermartingale.
        By Doob's Theorem, we have
        \begin{equation*}
     P^N_{k,1}\left(\tau_{\geq \varepsilon N}\geq \tau_0, \tau^\star\geq \tau_0\right)=P^N_{k,1}\left(\tau_0\leq\tau  \right)=E^N_{k,1}[M_{\tau_0 \wedge \tau}]\leq E^N_{k,1}[M_0]=e^{-c},
     \end{equation*}
concluding the proof.
    \end{proof}
To prove Lemma \ref{prop2}, we use the following result. As long as~$X_t$ remains large enough and~$Y_t$ sufficiently small, then~$Y_t$ has a positive drift (as explained above) and thus comes back above~$\varepsilon N$ an exponentially large number of times.  For any~$j\in \mathbb{N}$, let us call~$\tau_j^+=\inf\{t\geq 1: Y_t=j\}$ the first return time to~$j$.
\begin{lemma} \label{propmin}
   For any~$\gamma>\frac{1}{2},\, \delta \in (0,\gamma-\frac{1}{2})$ and for any~$\varepsilon \in (0,\gamma-\delta-\frac{1}{2})$, there exists~$c>0$ such that for any~$N$ large enough and for any~$k\geq  (\gamma-\delta) N$,
    \begin{equation*}
        P^N_{k,\lceil \varepsilon N\rceil}\big(\tau_0<\min\{\tau^\star, \tau_{\lceil\varepsilon N\rceil}^+\}\big)<e^{-c N}.
    \end{equation*}
\end{lemma}
\begin{proof}[Proof of Lemma \ref{propmin}]
   Let~$\gamma>\frac{1}{2},\, \delta \in (0,\gamma-\frac{1}{2})$ and~$\varepsilon \in (0,\gamma-\delta-\frac{1}{2})$. Let us define~$j:=\lceil\varepsilon N\rceil$. First, since~$Y_t$ can decrease at each step by at most~$1$, and given the different possibilities at the first step, we have
   \begin{equation*}
       P_{k,j}^N\big(\tau_0<\min\{\tau^\star,\, \tau_j^+\}\big)\leq \max_{k'\in \{k, k-1\}}P_{k',\,j-1}^N\big(\tau_0<\min\{\tau^\star,\, \tau_{\geq j}\}\big).
   \end{equation*}
   Let~$k'\in \{k,\,k-1\}$. 
   Using the same arguments as in the proof of the previous lemma, we can show that when~$X_0=k', Y_0= j-1$ and~$c$ is small enough, the process 
    \begin{equation*}
 M_t=e^{-c Y_t}\mathbb{1}(t<\tau) \ \text{with} \ \tau=\min\{ \tau^\star, \tau_{\geq j}\}
\end{equation*}
    is a supermartingale with respect to the filtration~$(\mathcal{F}_t)_{t\geq 0}$ generated by~$(X_t, Y_t)_{t\geq 0}$. By Doob's Theorem, we have 
    \begin{equation*}
 P^N_{k', j-1}\left( \tau_0<\tau\right)=E_{k',j-1}^N[M_{\tau_0\wedge \tau}]\leq E^N_{k', j-1}[M_0]= e^{-c (\lceil\varepsilon N \rceil-1)},
\end{equation*}
and this concludes the proof.
\end{proof}
Lemma \ref{propmin} enables us to prove that if~$X_0\geq(\gamma-\delta)N$ and~$Y_0\geq \varepsilon N$, with high probability after stabilisation less than~$(\gamma-\delta)N$ particles remain.

\begin{proof}[Proof of Lemma \ref{prop2}]
Let~$\gamma>\frac{1}{2},\, \delta \in (0,\gamma-\frac{1}{2})$ and~$\varepsilon \in (0,\gamma-\delta-\frac{1}{2})$. Let~$k\geq (\gamma-\delta)N$ and let us call
\begin{equation}
    L=\sum_{t=1}^{\tau_0}\mathbb{1}(Y_t=\lceil \varepsilon N\rceil).
\end{equation}
Let~$(Z_t)_{t\geq 0}$ be an i.i.d.\ sequence of uniform variables in~$\{1,\,\ldots,\,N\}$ and assume that the microscopic chain is constructed by moving the active particle to the site~$Z_t$ at each step. For every~$t\geq 1$ let
\begin{equation*}
    S_t\ =\ \sum_{i=1}^t\mathbb{1}(Z_i=s_N)\,,
\end{equation*}
which is the number of times that the sink is chosen before time~$t$. Then, for any positive constant~$c>0$, we have the inclusion
\[\big\{L>e^{cN},\,\tau^\star>\tau_0\big\}\ \subset\ \big\{S_{\lfloor e^{cN}\rfloor}<N \big\}\,.\]
Note that~\smash{$S_{\lfloor e^{cN}\rfloor}$} is a binomial random variable with parameters~$(\lfloor e^{cN}\rfloor, \frac{1}{N})$. By  Hoeffding's inequality~\cite{W}, if~$X$ is a binomial~$(n,p)$ and if~$\ell \leq np$, then
   \begin{equation*}
       P(X\leq \ell)\leq \exp\left(-2n \left(p-\frac{\ell}{n}\right)^2\right).
   \end{equation*}
   Thus, for any~$N$ large enough,
   \begin{equation*}
  P^N_{k,\lceil \varepsilon N\rceil}\left(\tau^\star >\tau_0, L>e^{cN}\right)    \leq  P^N_{k,\lceil \varepsilon N\rceil}\left(S_{\lfloor e^{cN}\rfloor}<N\right)\leq \exp\left(-\frac{e^{cN}}{N}\right)\leq e^{-c'N}
   \end{equation*}
for a certain~$c'>0$.
Let~$c''>0$ be the constant given by Lemma \ref{propmin}. Taking~$c<c''/2$, we have
    \begin{align*}
         P^N_{k, \lceil \varepsilon N \rceil}\left(\tau^\star>\tau_0, L\leq e^{cN}\right)&=\sum_{n=0}^{\lfloor e^{cN} \rfloor}P^N_{k, \lceil \varepsilon N\rceil}\left(\tau^\star>\tau_0, L=n\right)\\&\leq \sum_{n=0}^{\lfloor e^{cN}\rfloor} %P^N_{k, \lceil \varepsilon N \rceil}\left(\tau_{\lceil\varepsilon N\rceil}^+<\min \{\tau_0,\tau^\star\}\right)^n
         \max_{k'>\left( \gamma-\delta \right)N}P^N_{k',\lceil \varepsilon N \rceil}\big(\tau_0<\min\{\tau^\star, \tau_{\lceil \varepsilon N\rceil}^+\}\big)
         \\&\leq (e^{cN}+1)e^{-c'' N}=e^{-(c''/2) N},
    \end{align*}
for~$N$ large enough, where we used Lemma \ref{propmin} in the last inequality.
\end{proof}

\section{Proof of the lower bound in Theorem \ref{mainthm2}}\label{sec-proof-lowerbound}
The lower bound in Theorem \ref{mainthm2} is a consequence of Lemma \ref{thmreturntime} with~$\gamma<\frac{1}{2}$. To use this lemma, in Sections \ref{sec-H1-holds} and \ref{sec-H2-holds} respectively, we check that Conditions \ref{hpH1} and \ref{hpH2} hold for the macroscopic chain~$(\xi_t)_{t\geq 0}$. 

By Lemma~\ref{conn-micr-macr}, one step in the macroscopic chain~$(\xi_t)_{t\geq 0}$ corresponds to the stabilisation of the microscopic chain~$(X_t,Y_t)$. We will rely on the following result about the stabilisation time~$\tau_0$ of~$(X_t,Y_t)_{t\geq 0}$.

\begin{lemma}\label{prop4}
   For any~$\gamma<\frac{1}{2}$, there exists~$c >0$ such that for any~$N$ large enough, for any~$k\leq \gamma N$, for any~$j\in \{1,\,\dots,\,k\}$ and for any~$n>\frac{4j}{1-2\gamma}$,
    \begin{equation*}
        P^N_{k,j}(\tau_0>n)\leq e^{-c n}. 
    \end{equation*}
\end{lemma}
We will use this lemma twice. First, in Section \ref{sec-H1-holds}, we show that when the
particle density is below~$1/2$, each step of the macroscopic chain does
not last very long (using Lemma~\ref{prop4} with~$j=1$) and thus we expect that
only few particles can escape.
Second, in Section \ref{sec-H2-holds}, using that the lemma holds for all~$j$, we
show that, if during
the evolution of the microscopic chain the density goes below~$\gamma$, then
whatever the fraction of active particles is, it quickly decreases to~$0$.
\begin{proof}[Proof of Lemma \ref{prop4}]
Let $\gamma \in (0,\frac{1}{2}),\,\varepsilon:=\frac{1}{2}-\gamma,\,k\leq \gamma N$ and $j\in \{1,\dots,\,k\}$.
   Let us call~\smash{$\Tilde{Y}_t:=Y_t+\varepsilon t$}.  Then, for any~$n\geq 1$, 
\begin{align*}
    P^N_{k,j}(\tau_0>n)&=P^N_{k,j}(\forall\, 0\leq s\leq n, \,Y_{s}>0)\\&= P^N_{k,j}(\forall\, 0\leq s\leq n, \,\Tilde{Y}_s>\varepsilon s, Y_s>0)\\&%\leq P^N_{k,j}(\Tilde{\tau}_{n \varepsilon }< \min\{\Tilde{\tau}_0, \tau_0\})
    \leq P_{k,j}^N(\Tilde{\tau}_{\geq n \varepsilon}<  \tau_0),
\end{align*}
where~$\Tilde{\tau}_{\geq \ell}=\inf\{t\geq 1: \Tilde{Y}_t\geq \ell\}$, for any~$\ell\geq 0$. To conclude, it is enough to prove that there exists~$c_2>0$ such that~$P_{k,j}^N(\Tilde{\tau}_{\geq n\varepsilon }<\tau_0)<e^{-c_2 n}$.
Let us call~$M_t=e^{\bar c \Tilde{Y}_t}\mathbb{1}(t<\tau_0)$ where~$\bar c>0$ is a constant which will be chosen later and let~$(\mathcal{F}_t)_{t\geq 0}$ be the filtration generated by~$(X_t,Y_t)_{t\geq 0}$. Then, 
\begin{equation*}
    E^N_{k,j}\big[M_{t+1}\mid \mathcal{F}_t\big]\leq M_t E^N_{k,j}\big[e^{\bar{c}\Delta \Tilde{Y}_t}\mathbb{1}(t+1< \tau_0)\mid \mathcal{F}_t\big] \leq M_t \mathbb{1}(t<\tau_0) e^{\bar c \varepsilon}
E_{k,j}^N\big[e^{\bar c\Delta Y_t}\,\big|\,\mathcal{F}_t\big],
\end{equation*}
where, by \eqref{jump-to-sleeping}, we have 
\begin{align*}
    \mathbb{1}(t<\tau_0) e^{\bar c \varepsilon}
E_{k,j}^N\big[ e^{\bar c\Delta Y_t}\,\big|\,\mathcal{F}_t \big]
\ \leq\
e^{\bar c \varepsilon}
\Big( e^{-\bar c} + \frac{X_t-Y_t}{N} (e^{\bar c}-e^{-\bar c}) \Big)
\ \leq\
e^{\bar c \varepsilon}
\Big( e^{-\bar c} + \gamma (e^{\bar c}-e^{-\bar c}) \Big).
\end{align*}
If we choose~$\bar c>0$ small enough such that 
\begin{equation*}
 e^{\bar c \varepsilon}
\Big( e^{-\bar c} + \gamma (e^{\bar c}-e^{-\bar c}) \Big)
\ \leq\ 1, 
\end{equation*}
which is possible because the left-hand side above
is~$1-\bar c\varepsilon + o(\bar c)$ when~$\bar c$ tends to~$0$, we have that~$M_t$ is a supermartingale with respect to~$(\mathcal{F}_t)_{t\geq 0}$. Let us call~$\tau=\min\{ \tau_0,\, \Tilde{\tau}_{\geq n \varepsilon} \}$. By Doob's Theorem, for any~$ j\in \{1,\,\dots,\, k\}$ we have
\begin{equation*}
    e^{j\bar{c}}=E^N_{k,j}[M_0]\geq E^N_{k,j}[M_\tau]\geq P^N_{k,j}(\Tilde{\tau}_{\geq n \varepsilon}< \tau_0) e^{\bar{c}\varepsilon n}.
\end{equation*}
Choosing~$n \geq 2j/\varepsilon$ gives the desired result, with~$c=\bar
c\varepsilon/2$.

\end{proof}

\subsection{Positive drift for densities below \texorpdfstring{$1/2$}{1/2}} \label{sec-H1-holds}
 We now verify that Condition \ref{hpH1} holds for the Markov chain~$(\xi_t)_{t\geq 0}$. 
The idea is that, when the initial number of particles is sufficiently small, we can obtain an upper bound of the number of exiting particles in a single step. In particular, if~$\xi_0<\frac{N}{2}$ then~$(\xi_t)_{t\geq 0}$ has a positive drift. In the following result, we adopt the convention that a geometric random variable takes values
in~$\mathbb{N}\setminus\{0\}$.
\begin{lemma}\label{prop6}
For any~$\gamma<\frac{1}{2}$, for every~$k\leq \gamma N$, under~$\mathbb{P}_k^N$ we have
that~$2-\Delta \xi_0$ is stochastically dominated by a geometric random variable with parameter~$\frac{2}{3}$.
\end{lemma}
Letting~$G$ denote the geometric random variable, this lemma shows that Condition \ref{hpH1}
holds
with~$D=2-G$, with~$\mathbb{E}[D]=2-3/2>0$.
\begin{proof}[Proof of Lemma \ref{prop6}]
We want to show that~$\Delta \xi_0$ stochastically dominates~$2-\text{Geom}(\frac{2}{3})$.
Let~$(Z_t)_{t\geq 0}$ be an i.i.d.\ sequence of uniform variables in~$\{1,\,\ldots,\,N\}$ and assume that the microscopic chain evolves by moving the active particle to the site~$Z_t$ at each step. For every~$t\geq 1$ let \[S_t\ =\ \sum_{i=1}^t\mathbb{1}(Z_i=s_N)\,,\] which is the number of times that the sink is chosen before time~$t$. 
  
Let~$c$ be the constant given by Lemma~\ref{prop4} and consider~$k\leq \gamma N$.  By Lemma~\ref{conn-micr-macr}, $1-\Delta \xi_0$, the number of exiting particles in one step of the macroscopic chain~$(\xi_t)_{t\geq 0}$ is equal to~$S_{\tau_0}$, the number of exiting particles during stabilisation of the microscopic chain~$(X_t, Y_t)_{t\geq 0}$. We deduce that, for any~$ \ell\geq 1$ and for any~$a\geq \frac{4}{1-2\gamma}$, 
\begin{align*}\mathbb{P}_k^N\left(\Delta \xi_0\leq 1-\ell\right)= P^N_{k+1,1}(S_{\tau_0}\geq \ell)&\leq P^N_{k+1,1}(\tau_0>a \ell)+P^N_{k+1,1}(\tau_0\leq a \ell, S_{\tau_0}\geq \ell)\\&\leq e^{-c a \ell}+P^N_{k+1,1}( S_{a \ell}\geq  \ell),
\end{align*}
where we used Lemma \ref{prop4} in the last inequality.

Using a Chernoff bound, for any~$\lambda>0$, we have that 
\begin{align*}
P^N_{k+1,1}(S_{a \ell}\geq  \ell)\leq e^{-\lambda \ell}E^N_{k+1,1}[e^{\lambda S_{a\ell}}]&=\exp\left( a \ell\ln \left(\frac{e^{\lambda}-1}{N}+1\right)-\lambda \ell\right)\\& \leq \exp\left( a \ell \frac{e^\lambda}{N}-\lambda \ell\right).
%\\&\leq \exp\left(  \ell -\ell \ln \left( \frac{N}{a}\right)\right)\leq \frac{1}{3^\ell}.
\end{align*}
Taking~$N>a$ and choosing~$\lambda=\ln(N/a)$, we get
\begin{equation*}
    P^N_{k+1,1}(S_{a \ell}\geq  \ell) \leq \exp\left(  \ell -\ell \ln \left( \frac{N}{a}\right)\right)\leq \frac{1}{2}\frac{1}{3^\ell}, 
\end{equation*}
for~$N$ large enough.
Thus, for any~$\ell \geq 1$, taking~$a$ large enough and~$N$ large enough, we get
\begin{align*}
   \mathbb{P}^N_{k}(2-\Delta \xi_0\geq 1+ \ell)= \mathbb{P}^N_{k}(\Delta \xi_0\leq 1-\ell)\leq e^{-ca\ell}+\frac 1
2 \frac
1 {3^\ell} \leq \frac 1 {3^\ell}.
\end{align*}
  We conclude that~$2-\Delta \xi_0$ is stochastically dominated by a geometric random variable with parameter~$\frac{2}{3}$,  concluding the proof.
\end{proof}
\subsection{Few particles exit in one step} \label{sec-H2-holds}
We now show that Condition \ref{hpH2} is satisfied by proving that, for any~$0<\beta<\gamma<\frac{1}{2}$, when~$\xi_0\geq \gamma N$, the probability that after one step we have less than~$\beta N$ particles is exponentially small.   
\begin{lemma}\label{lemmaxi1}
    For any~$0<\gamma<\frac{1}{2}$ and for any~$0<\beta<\gamma$, there exists~$c>0$ such that, for any~$N$ large enough and for any~$k\geq \gamma N$,
    \begin{equation*}
\mathbb{P}^N_k\left(\xi_1< \beta N\right)\leq e^{-cN}.
    \end{equation*}
\end{lemma}

\begin{proof}[Proof of Lemma \ref{lemmaxi1}]
Let $k_1:=\lceil\beta N\rceil$ and~$ k_2:=\lceil\gamma N\rceil$.
Let us consider the microscopic chain~$(X_t, Y_t)_{t\geq 0}$ with~$X_0\geq k_2$. 
By Lemma \ref{conn-micr-macr} we have that, for any~$k\geq k_2$, 
    \begin{equation*}
    \mathbb{P}_k^N\left(\xi_1< k_1\right)=
    P^N_{k+1,1}\left(X_{\tau_0}< k_1\right).
    \end{equation*}
    Using that~$X_t$ can decrease by at most~$1$ at each step of the microscopic chain, the strong Markov property applied at the first time when~$X_t=k_2-1$ tells us that 
    \begin{align*}
      P^N_{k+1,1}\left(X_{\tau_0}< k_1\right)  &\leq \max_{j< k_2}P_{k_2-1,j}^N\left(X_{\tau_0}< k_1\right)\\&=\max_{j < k_2}P_{k_2-1,j}^N\left(X_{\tau_0}-X_0-1 < k_1-k_2\right)\\&\leq \max_{j<k_2} P_{k_2-1,\,j}^N\left(1+X_0-X_{\tau_0}>(\gamma-\beta) N\right).
    \end{align*}
Let~$(Z_t)_{t\geq 0}$ be an i.i.d.\ sequence of uniform variables in~$\{1,\,\ldots,\,N\}$ and assume that the microscopic chain evolves by moving the active particle to the site~$Z_t$ at each step. For every~$t\geq 1$, we denote by~$S_t$ as the number of exiting particles up to time~$t$, that is, \[S_t\ =\ \sum_{i=1}^t\mathbb{1}(Z_i=s_N)\,.\] 
Let~$c>0$ be the constant given by Lemma \ref{prop4}, let~$j< k_2$ and~$a>\frac{4\gamma}{1-2\gamma}$, so that~$a N\geq \frac{4 k_2}{1-2\gamma}$ for $N$ large enough. Recalling that~$X_0-X_{\tau_0}$ is equal to~$S_{\tau_0}$, the number of particles exiting during stabilisation in the microscopic chain~$(X_t, Y_t)_{t\geq 0}$, we deduce that,
\begin{align*}P_{k_2-1,\,j}^N&\left(1+X_0-X_{\tau_0}>(\gamma-\beta) N\right)\\&= P^N_{k_2-1,\,j}(1+S_{\tau_0}> (\gamma-\beta)N)\\&\leq P^N_{k_2-1,\,j}(\tau_0> a N)+P^N_{k_2-1,\,j}(\tau_0\leq  a N, 1+S_{\tau_0}>  (\gamma-\beta)N)\\&\leq e^{-c a N}+P^N_{k_2-1,\,j}(1+S_{a N}> (\gamma-\beta)N),
\end{align*}
where we used Lemma \ref{prop4} in the last inequality.

Using a Chernoff bound similar to that used in the proof of Lemma \ref{prop6}, we
get~$c'>0$ such that, for~$N$ large enough
\begin{align*}
P^N_{k_2-1,j}(1+S_{a N}>  (\gamma-\beta)N)
\leq e^{-c'N}.
\end{align*}
Thus, we can find a constant~$c''>0$ such that
\begin{align*}
    P^N_{k_2-1,j}\left(1+X_0-X_{\tau_0}>(\gamma-\beta) N\right)\leq e^{-c'' N},
\end{align*}
 concluding the proof.
\end{proof}

\appendix
\section{Appendix: proof of Lemma \ref{thmreturntime}}\label{sec-appendix}
Let~$N\geq 1$, consider an irreducible, time-homogeneous Markov chain~$(Z_t)_{t\geq 0}$ with state space~$\{0,\,\dots,\,N-1\}$ and  let us denote by~$\nu^N$ its stationary distribution. Let~$\mathbb{P}_k^N, \mathbb{E}_k^N$ be the distribution and the expectation of~$(Z_t)_{t\geq 0}$ when~$Z_0=k$ respectively and let~\smash{$T_k^+=\inf\{t\geq 1: Z_t=k\}$} denote the first return time to~$k$.  Let~$0<\alpha<\beta<\gamma \leq 1$  and let us define~$k_0:=\lfloor\alpha N\rfloor,\, k_1:=\lceil\beta N \rceil$ and $k_2:=\lceil\gamma N\rceil$. For any~\smash{$j\in \{0,\,\dots,\,N-1\}$}, let~$T_{\geq j}$ and $T_{\leq j}$ be the first hitting time of the intervals~$[j,N]$ and~$[0,j]$ respectively. When~\smash{$Z_0\geq j$}, let us call~$T_{\geq j}^+=\inf\{t\geq 1: Z_t \geq j\}$ the first return time to the interval~$[j,N]$. Suppose that Conditions \ref{hpH1} and \ref{hpH2} hold. 

Since~$\nu^N$ is the stationary distribution of an irreducible chain, we have that, for any~$k \in \{0,\,\dots,\,N-1\}$, \smash{$\nu^N(k)=\big(\mathbb{E}_k^N[T_k^+]\big)^{-1}$} (see~\cite{LPW} for example). Thus, it is enough to prove that there exists~$c>0$ such that if~$k\leq k_0$, 
\begin{equation} \label{eqexptime}
    \mathbb{E}_k^N[T_k^+]\geq e^{cN}.
\end{equation}
Indeed, if the above bound is satisfied, then we can find~$c'>0$ such that, for any~$N$ large enough,
\begin{equation*}
    \nu^N\left(\{0,\,\dots,\, k_0\}\right)\leq \sum_{k=0}^{k_0}e^{-cN}\leq e^{-c'N}.
\end{equation*}
To prove \eqref{eqexptime}, we first show that there exists~$C>0$ such that, for any~$N$ large enough and for any~$k\leq k_0$, 
\begin{equation}\label{eqtime1}
    \mathbb{P}_k^N\left(T_{\geq k_2}<T_k^+\right)>C.
\end{equation}
Then, we prove that there exist~$c_0, C'>0$ such that for any~$N$ large enough and for any~$k'\geq k_2$, 
\begin{equation}\label{eqtime2}
    \mathbb{P}_{k'}^N\left(T_{\leq k_0}>e^{c_0 N}\right)>C'.
\end{equation}
Combining \eqref{eqtime1} and \eqref{eqtime2}, we can find~$0<c_1<c_0$ such that, for any~$N$ large enough and for any~$k\leq k_0$, we have 
\begin{align*}
    \mathbb{E}_k^N[T_k^+]&\geq \mathbb{P}_k^N\left(T_{\geq k_2}< T_k^+\right)\inf_{k'\geq k_2} \mathbb{E}_{k'}^N[T_k]
    \\& \geq \mathbb{P}_k^N\left(T_{\geq k_2}< T_k^+\right)\inf_{k'\geq k_2} \mathbb{E}_{k'}^N[T_{\leq k_0}]
    \\&\geq e^{c_0 N}\mathbb{P}_k^N\left(T_{\geq k_2}< T_k^+\right) \inf_{k'\geq k_2}\mathbb{P}_{k'}^N\left(T_{\leq k_0}>e^{c_0 N}\right)>e^{c_1 N}, 
\end{align*}
which implies \eqref{eqexptime}.
Thus, there remains to show \eqref{eqtime1} and \eqref{eqtime2}. 

Let us start with \eqref{eqtime2}.
If~$Z_0\geq k_2$, the random variable~$T_{\leq k_0}$ dominates a geometric random variable with parameter~\smash{$\max_{k'\geq k_2} \mathbb{P}_{k'}^N\big(T_{\leq k_0}<T^+_{\geq k_2}\big)$}. Thus, we show that there exists~$c_2>0$ such that, for any~$N$ large enough,
\begin{equation}\label{eqmax2}
   \max_{k'\geq k_2} \mathbb{P}_{k'}^N\big(T_{\leq k_0}<T^+_{\geq k_2}\big)\leq e^{-c_2 N}.
\end{equation}

 Let~$c>0$ be the constant given by Condition \ref{hpH2}. Using this 
 condition, for any~$k'\geq k_2$, we write
\begin{align*}\label{eq3.1}
    \mathbb{P}_{k'}^N\big(T_{\leq k_0}<T^+_{\geq k_2}\big)&\leq \mathbb{P}_{k'}^N\left(Z_1<k_1\right)+\max_{k''\in [k_1,k_2-1]}\mathbb{P}_{k''}^N\big(T_{\geq k_2}>T_{\leq k_0}\big)\\&\leq e^{-cN}+\max_{k''\in [k_1,k_2-1]}\mathbb{P}_{k''}^N\big(T_{\geq k_2}>T_{\leq k_0}\big).
\end{align*}
To conclude the proof of \eqref{eqmax2}, it is enough to prove that there exists~$c_3>0$ such that, for every~$N$ large enough and for every~$k''\in [k_1,k_2-1]$,
\begin{equation}\label{eq5.1}
    \mathbb{P}_{k''}^N\big(T_{\geq k_2}>T_{\leq k_0}\big)\leq e^{-c_3 N}.
\end{equation}

To prove \eqref{eq5.1}, we show that if~$a$ is small enough and if~$Z_0 \in [k_1,k_2-1]$, the process
$M_t=e^{-a Z_t} \mathbb{1}(t<T_{\geq k_2})$ is a supermartingale with respect to the filtration~$(\mathcal{F}_t)_{t\geq 0}$ generated by~$(Z_t)_{t\geq 0}$. 
 For any~$t\geq 0$ and for any~$k''\in [k_1,k_2-1]$, we write
 \begin{equation*}
   \mathbb{E}^N_{k''}[M_{t+1}\mid \mathcal{F}_t]\leq M_t \mathbb{E}^N_{k''}[e^{-a\Delta Z_t}\mathbb{1}(t+1<T_{\geq k_2})\mid \mathcal{F}_t]\leq M_{t}\max_{k'''< k_2}\mathbb{E}_{k'''}^N[e^{-a\Delta Z_0}].
   \end{equation*}

 Let~$D$ be the random variable given by Condition \ref{hpH1}, such that~$\Delta Z_0$ stochastically dominates~$D$ and~$\mathbb{E}[D] > 0$. Then, defining~$f(a):=\mathbb{E}[e^{-a D}]$, we have~$f(0)=1$ and~\smash{$f'(0)<0$}, which implies that there exists~$a_0>0$ such that~$f(a_0)\leq 1$. In particular, for any~$k'''<k_2$, we have that~$\mathbb{E}_{k'''}^N[e^{-a_0 \Delta Z_0}]\leq 1$.
    We deduce that the process~$M_t=e^{-a_0 Z_t} \mathbb{1}(t<T_{\geq k_2})$ is a supermartingale with respect to~$(\mathcal{F}_t)$.
 By Doob's Theorem, for any~$k''\in [k_1, k_2-1]$ we have that
\begin{align*}
    \mathbb{P}_{k''}^N\big(T_{\geq k_2}>T_{\leq k_0 } \big)&=\mathbb{P}_{k''}^N\big(T_{\geq k_2}\wedge T_{\leq k_0 }=T_{\leq k_0} \big)\leq \mathbb{E}_{k''}^N\big[M_{T_{\geq k_2} \wedge T_{\leq k_0 }}\big]e^{a_0 k_0}\\&\leq \mathbb{E}_{k''}^N[M_{0}]e^{a_0 k_0}\leq e^{-a_0(k_1-k_0)}\leq e^{-a_0(\beta-\alpha)N}.
\end{align*}
This concludes the proof of \eqref{eqtime2}. It remains to prove \eqref{eqtime1}. 
Let~$k\leq k_0$ and let~$D$ be the random variable given by Condition \ref{hpH1}.  
  Then, there exists~$C''>0$ such that
  \begin{equation*}
\mathbb{P}^N_k(Z_1>k)=\mathbb{P}^N_k(Z_1-k\geq 1)\geq P(D>0)>C''.
 \end{equation*}
  Moreover, using exactly the same argument as is used for the
 proof of \eqref{eqtime2}, we can find~$c_4>0$ such that, for any~$k'> k$, 
\begin{equation*}
    \mathbb{P}^N_{k'}(T_{\geq k_2}<T_k)= 1-\mathbb{P}^N_{k'}(T_{\geq k_2}>T_k)\geq 1-e^{-c_4{k'}+c_4 k}
 \geq 1-e^{-c_4}.
\end{equation*}
Thus, for any~$k\leq k_0$, we have that
\begin{equation*}
    \mathbb{P}^N_k(T_{\geq k_2}<T^+_k)\geq \mathbb{P}^N_k(Z_1>k)\inf_{k'>k} \mathbb{P}^N_{k'}(T_{\geq k_2}<T_k)\geq C'' (1-e^{-c_4}),
\end{equation*}
concluding the proof.


\begin{thebibliography}{99}

\bibitem{BTW} Bak, P., Tang, C., Wiesenfeld, K.:
Self-organized criticality: An explanation of the 1/f noise.
\emph{Phys. Rev. Lett.} \textbf{59} (1987), no.~4, 381--384.

\bibitem{CFT} Campailla, C., Forien, N., Taggi, L.:
The critical density of the Stochastic Sandpile Model,
(2024), \ARXIV{2410.18905}

\bibitem{FLW2} Fey, A., Levine, L., Wilson, D. B.:
Driving sandpiles to criticality and beyond.
\emph{Phys. Rev. Lett.} \textbf{104} (2010), no.~14, 145703.

\bibitem{FM} Fey, A., Meester, R. W. J.:
Critical densities in Sandpile Models with quenched or annealed disorder.
\emph{Markov Processes and Related Fields} \textbf{21} (2015), 57--83. \MR{3381224}.

\bibitem{W} Hoeffding, W.:
Probability Inequalities for Sums of Bounded Random Variables.
\emph{Journal of the American Statistical Association} \textbf{58} (1963), no.~301, 13--30.

\bibitem{HHRR22} Hoffman, C., Hu, Y., Richey, J., Rizzolo, D.:
Active Phase for the Stochastic Sandpile on $\mathbb{Z}$, (2022), \ARXIV{2212.08293}

\bibitem{HJJ} Hoffman, C., Johnson, T., Junge, M.:
The density conjecture for activated random walk, (2024), \ARXIV{2406.01731}


\bibitem{JMT} Járai, A., Mönch, C., Taggi, L.:
The critical window in activated random walk on the complete graph,(2024), \ARXIV{2304.10169}

\bibitem{LPW} Levin, D. A., Peres, Y.:
Markov Chains and Mixing Times, 2nd ed. With contributions by Elizabeth L. Wilmer,
With a chapter on ``Coupling from the past'' by James G. Propp
              and David B. Wilson.
\emph{American Mathematical Society}, Providence, RI, 2017. \MR{3726904}

\bibitem{LL} Levine, L., Liang, F.:
Exact sampling and fast mixing of activated random walk.
\emph{Electron. J. Probab.} \textbf{29} (2024), Paper No. 184, 20 pp.
 \MR{4838433}

\bibitem{LS24} Levine, L., Silvestri, V.:
Universality conjectures for activated random walk.
\emph{Probab. Surv.} \textbf{21} (2024), 1--27.
 \MR{4718500}


\bibitem{R} Rolla, L. T.:
Activated random walks on $\mathbb{Z}^d$.
\emph{Probab. Surv.} \textbf{17} (2020), 478--544. \MR{4152668}

\bibitem{RS} Rolla, L. T., Sidoravicius, V.:
Absorbing-state phase transition for driven-dissipative stochastic dynamics on $\mathbb{Z}$.
\emph{Invent. Math.} \textbf{188} (2012), no.~1, 127--150. \MR{2897694}

\bibitem{ST17} Sidoravicius, V., Teixeira, A.:
Absorbing-state transition for stochastic sandpiles and activated random walks.
\emph{Electron. J. Probab.} \textbf{22} (2017), Paper No. 33, 35 pp. \MR{3646059}


\end{thebibliography}
\end{document}